\documentclass[a4paper, 10pt]{extarticle}
\usepackage[top=70pt,bottom=70pt,left=75pt,right=77pt]{geometry}

\usepackage{amssymb, amsmath, amsthm, setspace, bbm, prettyref, graphicx, float, subfigure, color, mathrsfs, mathtools}

\usepackage{thmtools}
\usepackage{Lutsko_Style}
\DeclareMathSizes{10}{9}{7}{5}
\usepackage[normalem]{ulem}

\makeatletter
\newtheoremstyle{mystyle}%
{\topsep}{\topsep}
{\itshape}{}
{\bfseries}{.}
{0.5em}
{\thmname{\@ifempty{#3}{#1}\@ifnotempty{#3}{#3}}}
\makeatother
\theoremstyle{mystyle}

\date{}

\title{Pair Correlation of the Fractional Parts of $\alpha n^{\theta}$}
 \author{
{\sc Christopher Lutsko,
Athanasios Sourmelidis,
and Niclas Technau
}
}

\begin{document}

  \maketitle
  \begin{abstract}
   \noindent
   
   Fix $\alpha,\theta >0$, and consider the sequence $(\alpha n^{\theta} \Mod 1)_{n\ge 1}$. Since the seminal work of Rudnick--Sarnak (1998), and due to the Berry--Tabor conjecture in quantum chaos, the fine-scale  properties of these dilated mononomial sequences have been intensively studied. In this paper we show that for $\theta <14/41$, and $\alpha>0$, the pair correlation function is Poissonian. While (for a given $\theta \neq 1$) this strong pseudo-randomness property has been proven for almost all values of $\alpha$, there are next-to-no instances where this has been proven for explicit $\alpha$. Our result holds for all $\alpha>0$ and relies solely on classical Fourier analytic techniques. This addresses (in the sharpest possible way) a problem posed by Aistleitner--El-Baz--Munsch (2021).

  \let\thefootnote\relax
  \let\thefootnote\relax 
  \end{abstract}

  \setlength{\abovedisplayskip}{1mm}

  \section{Introduction}

Let  $x=(x_n)_{n \ge 1} $
be a sequence on the unit interval $[0,1)$.
The \emph{pair correlation function} of $x$
measures the correlation between points in the initial segment 
$\{x_n:n \leq N \}$ on the scale of the mean spacing, $1/N$,
and is defined by
\begin{align}
R(x,N,f) := \frac{1}{N} \sum_{i\neq j \leq N} \sum_{k\in \Z} f(N(x_i-x_j+k)),
\end{align}
where $f \in \cC_c^\infty(\R)$ is a compactly supported, $C^\infty$-function.
The sequence $x$ is said to have \emph{Poissonian pair correlation} 
if the pair correlation function converges 
to the integral of $f$ (over $\mathbb{R}$)
as $N \to \infty$, just as one would expect for 
uniformly distributed and independent random variables. 
That is, the sequence $x$ has Poissonian pair correlation if for all $f \in \cC_c^\infty(\R)$
\begin{align}
 \lim_{N\rightarrow \infty} 
R(x,N,f) = \int_{\R}f(t) ~\mathrm{d}t.
\end{align}
The notion of Poissonian pair correlation defines a strong measure of pseudo-randomness 
and is a basic concept in quantum chaos.
Unsurprisingly, various efforts have been made
\cite{RudnickSarnak1998,BocaZaharescu2000, RudnickZaharescu2002,MarklofStrom2003, Heath-Brown2010,
  AistleitnerLarcherLewko2017}
to study the pair correlation function of monomial sequences
\begin{align}\label{def: monomial}
 (\alpha n^{\theta} \Mod 1)_{n\geq 1},
\end{align}
where $\theta>0$ and $ \alpha>0$. However, little progress has been made to verify that the pair correlation of such monomial sequences
is Poissonian (under explicit conditions on $\alpha, \theta$). 
We present the state of the art for \eqref{def: monomial} 
in Section \ref{ss: history}. In this paper we prove the first general and explicit result showing that such monomial sequences exhibit Poissonian pair correlation. Namely,

\begin{theorem}\label{thm: main theorem}
If $\theta\in (0, 14/41)$ and $\alpha>0$, then \eqref{def: monomial} has Poissonian pair correlation.
\end{theorem}
\begin{remark} 
    \begin{enumerate}
        \item Our method applies 
        to higher level correlations, 
        although this generalisation is 
        not straightforward as it requires a genuinely
         multidimensional approach
        (see \cite{LutskoTechnau2021},
        and also \cite{LutskoTechnau2022}). 
        Moreover, the only arithmetic input 
        of our method are exponential sum bounds.
        Thus, with some modification, 
        the method can be extended 
        to more general sequences satisfying 
        certain growth conditions. 
        \item The method of proof allows one to show that the pair correlation function converges to $\int_{\mathbb{R}} f(x) ~\mathrm{d}x$ with a polynomially decaying error in $N$ which is uniform for all $\alpha$ in a fixed compact interval.
        \item When $\theta =1/3$ and $\alpha^3 \in \Q$, then the triple correlation is not Poissonian (because of the cubes $n^3$). Thus, Theorem \ref{thm: main theorem} gives an example of a sequence whose pair correlation is Poissonian, but whose triple correlation is not.
    \end{enumerate}
\end{remark}

\noindent\textbf{Organization of the Paper:} Subsection \ref{ss: history} 
presents a brief history of these monomial sequences.
Subsection \ref{ss:Ideas} sketches the proof of Theorem \ref{thm: main theorem},
and Subsection \ref{ss:Heuristic} provides a heuristic argument which
indicates the limitations of our method. 
In Section \ref{s:Pair Correlation} we collect lemmata,
reducing matters to bounding certain exponential sums. 
Finally, in Section \ref{s:Proof of the Main Theorem} we prove Theorem \ref{thm: main theorem}.

\subsection{Background}\label{ss: history}

The study of monomial sequences dates back to Weyl \cite{Weyl1916} 
who used them in his study of uniform distribution (see \cite{KupiersNiederreiter1974} 
or \cite{DrmotaTichy2006}). 
More recently, there has been renewed interest in these sequences. 
In part, this is due to the well-known Berry--Tabor conjecture \cite{BerryTabor1977} 
which hypothesizes a link between the pseudo-randomness properties 
of energy levels, and dynamics of quantum systems. For more details,
see either of the following review papers \cite{Marklof2000, Rudnick2008}.

The holy grail of this field is to find circumstances 
for which a sequence has Poissonian gap statistics. 
That is, consider the distribution of gaps between neighboring first $N$ elements 
of the sequence -- scaled to have average $1$ -- then
we say the sequence exhibits Poissonian gap statistics
if this (finite) distribution converges to the exponential distribution,
as one would expect for independent random variables. 
While the aforementioned behavior is conjectured in many instances, 
it is truly challenging to prove. 
Thus, mathematicians have turned to weaker measures of pseudo-randomness. 
In particular, there has been a lot of recent work on the pair correlation. 
Indeed, if one could show that the $m$-level correlation converges 
to the expected value for independent random variables (for every $m \ge 2$) then, 
by the method of moments, one can infer that the sequence has Poissonian gap statistics.

If we consider the random variable counting the number of sequence elements in a randomly shifted set of size comparable to $1/N$, then the $m$-level correlations arise from the moments of this variable. Thus, the $m$-level correlations are natural measures
of pseudo-randomness in their own right. We refer to \cite{Marklof2007} for further discussion.

\vspace{-2mm}
\subsubsection{Pair correlation of deterministic sequences}\label{subsubsection: deterministic sequences}

The few deterministic sequences whose pair correlation functions
are known to be Poissonian either require the presence of particularly strong arithmetic structure,
or tools from homogeneous dynamics to apply. 
An example of the former is the work of Kurlberg and Rudnick \cite{RudnickKurlberg1999} 
on the (appropriately normalised) spacing of the quadratic residues of a highly composite 
modulus. 
In fact, they show that the gap statistics are Poissonian. 
However this setting requires the use of arithmetic tools which cannot be relied on in our situation.

On the homogeneous dynamics side, Elkies and McMullen \cite{ElkiesMcMullen2004}
established a remarkable link between 
\eqref{def: monomial}, for $(\theta,\alpha) = (1/2,1)$,
and flows on the modular surface $\SL_2(\R)/\SL_2(\Z)$.
They used this connection, and tools from homogeneous dynamics,
to establish that the corresponding gap distribution is 
\emph{not} Poissonian. 
Surprisingly, El-Baz, Marklof, and Vinogradov \cite{El-BazMarklofVinogradov2015a} 
then exploited said relationship further to show that, if one removed the squares, the pair correlation \emph{is Poissonian}. 
However, the connection to homogeneous dynamics requires a particular scaling property which only holds when $\theta=1/2$ 
and $\alpha^2 \in \Q$.
Indeed for $\alpha^2 \not\in \Q$ it is conjectured \cite{ElkiesMcMullen2004} 
the gap statistics are Poissonian.

For the sequence $(\alpha n \mod 1)_{n \ge 1}$, the three gap theorem 
(also known as the Steinhaus conjecture) states that the size of the 
gaps between neighboring points, at any time $N$, form a set of cardinality at most $3$. 
Hence, the local statistics are certainly not Poissonian. 
For background see \cite{MarklofStrom2017, MukherjeeKarner1998}.

\vspace{-4mm}

\vspace{2mm}

\subsubsection{Metric Poisson Pair Correlation}\label{subsubsection: metric theory}

Generally speaking, it is believed that, given a $\theta>0$, the pseudo-random properties 
of \eqref{def: monomial} are determined by the Diophantine properties of $\alpha$ 
(e.g see \cite[Remark 1.2]{RudnickSarnak1998}). However, 
in the absence of methods to prove Poissonian pair correlation for explicit values of 
$\alpha$, Rudnick and Sarnak \cite{RudnickSarnak1998} introduced the concept of
\emph{metric Poisson pair correlation}. Namely a general sequence $(x_n)_{n\geq 1}$ 
has metric Poisson pair correlation, if the dilated sequence 
$(\alpha x_n \Mod 1)_{n\geq 1}$ 
has Poissonian pair correlation for all $\alpha>0$ outside of a (Lebesgue) null set.  

For $\theta \in \N_{>1}$, Rudnick and Sarnak \cite{RudnickSarnak1998} 
proved the metric Poissonian pair correlation of \eqref{def: monomial} in the late 90s.
The case of non-integer $\theta >1$ was only recently settled 
by Aistleitner, El-Baz, and Munsch \cite{AistleitnerEl-BazMunsch2021}.
The regime $0<\theta <1$ was addressed 
by Rudnick and the second named author
\cite{RudnickTechnau2021}.

Special attention has been given to the quadratic case, $\theta = 2 $, due to 
its connection with quantum chaos and the boxed harmonic oscillator. 
Here, Heath-Brown \cite{Heath-Brown2010} 
gave an algorithmic construction of a dense set of $\alpha$
for which the pair correlation is Poissonian.
Moreover, there have been some results for longer-range correlations \cite{TechnauWalker2020, Lutsko2020}, convergence along sparse subsequences \cite{RudnickSarnakZaharescu2001,FassinaKimZaharescu2021}, and minimal gaps \cite{Zaharescu1995, Regavim2021, Rudnick2018}. However, finding explicit $\alpha$ for which the pair correlation is Poissonian remains out of reach.

Finally, it is worth noting that the metric Poisson pair correlation theory 
has been generalized beyond monomial sequences and exploits some deep connections 
to additive combinatorics \cite{AistleitnerLarcherLewko2017, BloomWalker2020}. 
However, this connection is beyond the scope of this paper.

\vspace{2mm}

    \vspace{2mm}
    \textbf{Notation:}
     Throughout, we use the usual Bachmann--Landau notation: for functions $f,g:X \rightarrow \mathbb{R}$, defined on some set $X$, we write $f \ll g$ (or $f=O(g)$) to denote that there exists a constant $C>0$ such that $\vert f(x)\vert \leq C \vert g(x) \vert$ for all $x\in X$.  
     Moreover let $f\asymp g$ denote $f \ll g$ and $g \ll f$. 
     Furthermore, let $f = o(g)$ denote that $\frac{f(x)}{g(x)} \to 0$. 
     
    Throughout we denote $e(x) = e^{2\pi i x}$ and $\wh{f}$ is the Fourier transform 
    (on $\R$) of $f$. 
    All of the sums which appear range over integers, in the indicated interval. 
    As $\alpha, \varepsilon, \theta, $ and $f$ are considered fixed, 
    we suppress any dependence in the implied constants. Moreover, for ease of notation, $\varepsilon>0$ may vary from line to line by a bounded constant.
    Further, we will frequently encounter the exponent 
    $$
    \Theta := \frac{1}{1-\theta}.
    $$

\subsection{Idea of the Proof}\label{ss:Ideas}
 The proof relies on a well-known Fourier decomposition. 
 First, we include the diagonal term in the pair correlation function, 
 to simplify technicalities. Thus, define
\begin{align}
  \wt{R}(N,f) : = \frac{1}{N} \sum_{\vect{y} \in [1,N]^2} \sum_{k \in \Z} f(N(\alpha y_1^{\theta}- \alpha y_2^{\theta} +k)).
\end{align}
Note that Theorem \ref{thm: main theorem} is equivalent to showing that (as $N \to \infty$)
\begin{align}
  \wt{R}(N,f) = \int_{\R}f(x) \, \mathrm{d}x + f(0) + o(1).
\end{align}
By the Poisson summation formula,
\begin{align}
  \wt{R}(N,f) = \wh{f}(0) +   \frac{1}{N^2} 
  \sum_{ \abs{k} \in [1, N^{1+\varepsilon}]} 
  \wh{f} \Big(\frac{k}{N}\Big) \abs{\sum_{y \in [1, N]} e(\alpha k y^{\theta})}^2   +  o(1)
\end{align}
for $\varepsilon >0$ where the $o(1)$-error comes from the fast decay of $\widehat{f}$. 
Note that $f$ can be decomposed 
into a sum of an even and an odd function. Further,
the Fourier coefficients of the odd part cancel out,
and the Fourier coefficients of the even part are even functions themselves. Thus, without loss of generality we may assume $f$ is even. Hence, it suffices to show that
\begin{align}\label{E def}
  \cE(N) := \frac{2}{N^2} \sum_{ k \in [1, N^{1+\varepsilon}]} 
  \wh{f} \Big(\frac{k}{N}\Big) \abs{ \sum_{y \in [1, N]} 
  e(\alpha k y^{\theta})}^2 = f(0)  + o(1).
 \end{align}

To achieve the desired bound requires a detailed analysis of the exponential sums in \eqref{E def}. 
We argue in, roughly, two steps:
first we decompose the innermost summation, 
and apply van der Corput's B-process to obtain a saving in the $y$-summation.
Second, we expand the square
and use some analytic tricks to reduce the estimates to exponential sums over $k$.
Now to obtain a saving in the $k$-summation,
we again use the B-process coupled with other estimates (such as Weyl differencing).

\subsection{Heuristic}
\label{ss:Heuristic}

After applying the B-process, interchanging the order of summation, 
extracting the main terms and dealing with the error terms, our task is
the following. We need to show that 
\begin{align*}
    \Err := \frac{2}{N^2} 
    \sum_{\substack{r_1,r_2 > cN^{\theta}\\r_1\neq r_2}}^N 
    \frac{1}{(r_1r_2)^{\frac{\Theta+1}{2}}}   \sum_{k \in [1,N^{1+\varepsilon}]} 
    \wh{f}\Big(\frac{k}{N}\Big) k^{\Theta} e(\gamma(\vect{r}) k^\Theta )
\end{align*}
is $o(1)$, as $N \to \infty$, where $\gamma(\vect{r}) = \beta(r_2^{1-\Theta}-r_1^{1-\Theta})$, 
and $\beta$ and $c$ depend only on $\theta$ and $\alpha$. 
Now we apply partial summation to reduce matters to estimating
\begin{align*}
    \abs{ \sum_{k \in [1, N^{1+\varepsilon}]} e(\gamma(\vect{r})k^\Theta)}.
\end{align*}
If we had square root cancellation for this sum
-- uniformly in $\gamma(\mathbf{r})$ --
then our method yields $\Err \approx N^{\theta-1/2+\varepsilon}$. 
In other words, even with optimal bounds, 
we cannot hope to go past the barrier $\theta = 1/2-\varepsilon$. 
To move past this barrier, our analytic method 
would require taking advantage of the cancellation 
between exponential sums for different values of $\vect{r}$.
This seems to be well beyond current technology.

Interestingly, if we consider instead the triple correlation function, 
the natural barrier to our methods turns out to be $\theta < 1/3$. 
In fact as we consider higher and higher correlations, that barrier goes to $0$.

\subsection{Preliminaries}
\label{ss:The A and B Processes}

The following two results are fundamental in the modern study of exponential sums. 
First, we recall an application of Weyl's differencing method 
(called the A-Process, see \cite[Theorem 2.9]{GrahamKolesnik1991}):

\begin{theorem}[$A$-Process]
\label{thm:A-Process}   
    Let $l\ge 0$ be an integer and let $M >2$. Suppose $\phi:[a,b)\to \R$ has $l+2$ continuous derivatives on $[a,b) \subseteq [M,CM)$, where $C>1$ is some fixed constant, and assume there exists a constant $F>0$ such that
    \begin{align}\label{growth in A-process}
        \phi^{(r)}(x) \asymp F M^{-r}
    \end{align}
    for $r = 1, \dots, l+2$. Then 
    \begin{align}\label{A-process conc}
        \sum_{x\in [a,b)} e(\phi(x)) \ll F^{1/(4L-2)} M^{1-(l+2)/(4L-2)} + F^{-1}M,
    \end{align}
    where $L := 2^l$. The implicit constant in \eqref{A-process conc} depends on the choice of $l$ and the implicit constant(s) in \eqref{growth in A-process}.
\end{theorem}

Further, we will use van der Corput's B-process, which follows from Poisson summation and a stationary phase argument (see {\cite[Theorem 8.16]{IwaniecKowalski2004}}):

\begin{theorem}[$B$-Process]
\label{thm:B-Process}Let $\phi:[A,B)\rightarrow\mathbb{R}$
be a $C^{4}$-function so that there are $\Lambda>0$ and $\eta\geq1$ with
\begin{equation}
\Lambda\leq\phi^{(2)}(x)<\eta\Lambda,
\qquad \big|\phi^{(3)}(x)\big|<\frac{\eta\Lambda}{B-A},
\qquad \big|\phi^{(4)}(x)\big|<
\frac{\eta\Lambda}{\left(B-A\right)^{2}}\label{eq: growth conditions in B Process}
\end{equation}
for all $x\in[A,B)$. Let $a=\phi'(A)$, and $b=\phi'(B)$. Then
\[
\sum_{n\in [A,B)}e(\phi(n))=e(1/8)\sum_{m\in[a,b)}\frac{e(\phi(x_{m})-mx_{m})}{\sqrt{\phi^{(2)}(x_{m})}}+\omega_{\phi}(A,B)
\]
where $x_{m}$ denotes the unique solution to $\phi'(x)=m$. Furthermore,
\begin{equation}
\omega_{\phi}(A,B)\ll\Lambda^{-\frac{1}{2}}+\eta^{2}\log(b-a+1),\label{eq: error bound in B Process}
\end{equation}
where the implied constant is absolute.
\end{theorem}

We will often need to bound weighted exponential sums.
To reduce these estimates to bounding unweighted sums, we use partial summation in the form of:

\begin{lemma}
\label{lem:partial summation}Let $(a_{s})_{s}$ and $(b_{s})_{s}$
be sequences of complex numbers. Fix a constant $c>1$.
If $T>0$ is such that $\left|b_{s}-b_{s+1}\right|\leq T/s$,
then
\[
\left|
\sum_{S\leq s<\tilde{S}}a_{s}b_{s}
\right| 
\leq\left(
\max_{S\leq s\leq cS}\left|b_{s}\right|
+O(T)\right)\max_{S\leq\tilde{S}\leq cS}\left|
\sum_{S\leq s<\tilde{S}}a_{s}\right|
\]
for any positive integers $S$ and $\tilde{S}$ 
satisfying $S\leq\tilde{S}\leq cS$.
\end{lemma}


\section{Reducing to Exponential Sum Bounds}
\label{s:Pair Correlation}

\subsection{Decomposing the sums and applying the B-Process}
\label{ss:Decomposing the sums}

Now consider the term $\cE(N)$, defined in \eqref{E def}.
We shall apply the B-process (Theorem \ref{thm:B-Process}) to the exponential sum in $\cE(N)$,
but presently we do not have sufficient control on the derivative of 
$y\mapsto \alpha k y^{\theta}$. To gain control, 
we use several 
decompositions. 
First we assume (w.l.o.g.) that $N = N_Q := Q^{\Gamma}$, for some fixed $\Gamma>0$,
 which will be chosen to be sufficiently large in our proof (in
a way depending on $\theta$), 
as it is enough to prove the correlations converge along such a subsequence, see {\cite[Lemma 3.1]{RudnickTechnau2020}}. 
Thus, we decompose 
the inner summation 
into the pieces
\[
E_q(k) : = \sum_{y \in [N_q, N_{q+1})} e(\alpha k y^{\theta})
\]
where $N_q:= q^{\Gamma}$. To catch the largest term, we set $N_{Q+1} = N_Q + 1$.
Thus, 
\begin{align}\label{Eq decom}
	\cE(N) = \frac{2}{N^2}\sum_{k \in [1,N^{1+\varepsilon})} 
	\wh{f} \Big(\frac{k}{N}\Big) \left|\sum_{q\leq Q}E_{q}(k)\right|^2.
\end{align}
Now, the next lemma shows that we can replace each $E_q(k)$ by 
\begin{align} \label{EB def}
  E^{(B)}_{q}(k) := c_1 \sum_{ r \in \cR_{q}(k)} \frac{k^{\Theta/2}}{r^{(\Theta+1)/2}} 
  e(\beta k^{\Theta}r^{1-\Theta}),
  \quad \mathrm{where}\,\, \cR_{q}(k) := \alpha \theta k(N_{q+1}^{\theta-1},  N_{q}^{\theta-1}],
\end{align}
which is the main term after applying the B-Process to $E_q(k)$; 
the constants $c_1$ and $\beta$ 
are defined by
$$
c_1 := e(-1/8)\sqrt{\Theta(\alpha\theta)^{\Theta}}, \qquad \mathrm{and} \qquad\,
\beta := \alpha^\Theta(\theta^{1-\Theta}-\theta^{\Theta}).
$$
For later reference let $\cR(k):=\bigcup_{q\in [1,Q]} \cR_q(k)$.
Let
\begin{align} \label{EB qdecom}
  \cE^{(B)}(N,\wh{f}): = \frac{2}{N^2}\sum_{\vect{q} \in [1,Q]^2} \sum_{k \in [1,N^{1+\varepsilon})} \wh{f} \left(\frac{k}{N}\right) E_{q_1}^{(B)}(k)\overline{E_{q_2}^{(B)}(k)}.
\end{align}

\begin{lemma}\label{lem:E B}
	Let $\cE(N)$ and  ${\cE^{(B)}(N,\wh{f})}$ be defined as in \eqref{Eq decom} and \eqref{EB qdecom} respectively. 
	Then
	\[
	\cE(N)={\cE^{(B)}(N,\wh{f})}+O\left(N^{-\theta+\varepsilon}+\left|{\cE^{(B)}(N,|\wh{f}|)}\right|^{1/2}N^{-\theta/2+\varepsilon}\right).
	\]

\end{lemma}

\begin{proof}
First we apply Theorem \ref{thm:B-Process} to each $E_q(k)$, $q\leq Q$, with 
   $\Lambda = k\alpha\theta(1-\theta)N_{q+1}^{\theta-2}$
and $\eta = 2^{5\Gamma}$. Hence,
  \begin{align*}
  \abs{E_q(k) - E_q^{(B)}(k)} \ll \frac{N_{q+1}}{\sqrt{N_{q+1}^{\theta}k}} + \log N \ll N^{1-\theta/2}k^{-1/2},
  \end{align*}
  with the implied constant being uniform in $q$ and $k$. 
  Thus
  \begin{align*}
    \cE(N)
    =  \frac{2}{N^2}\sum_{k \in [1,N^{1+\varepsilon})} \wh{f} \Big(\frac{k}{N}\Big) \left|\sum_{q\leq Q}E^{(B)}_{q}(k)+O\left(N^{1-\theta/2+
    {1/\Gamma}}k^{-1/2}\right)\right|^2.
  \end{align*}
  { Taking 
  $\Gamma>1/\varepsilon$ ensures that in the above error
  term $N^{1/\Gamma}$ can be replaced 
  by $N^{\varepsilon}$.}
  Squaring out and applying the Cauchy-Schwarz inequality yields now the lemma. 
\end{proof}

\subsection{The Diagonal}

Presently our goal is to establish \eqref{E def}. 
In view of Lemma \ref{lem:E B}, 
it suffices to prove that
\[
{\cE^{(B)}(N,\wh{f})}={f(0)}+o(1)\quad\text{ and }\quad{\cE^{(B)}(N,|\wh{f}|)={O}(1)},\quad N\to\infty,
\] 
for in that case \eqref{E def} is true and Theorem \ref{thm: main theorem} will follow.
{Computing the estimates of $\cE^{(B)}(N,\wh{f})$ and $\cE^{(B)}(N,|\wh{f}|)$ is done by completely analogous way and, therefore, we give detailed proofs only for $\cE^{(B)}(N,\wh{f})$.}

The main term, ${f(0)}$ will come from the diagonal term when expanding the square in ${\cE^{(B)}(N,\wh{f})}$.
That is, in \eqref{EB qdecom} we square out and consider the term
\begin{align*}
  {\cD(N,\wh{f})} := \frac{2\abs{c_1}^2}{N^2} \sum_{k \in [1,N^{1+\varepsilon})} 
  \wh{f}\Big(\frac{k}{N}\Big) \sum_{r \in \cR(k)} \frac{k^{\Theta}}{r^{\Theta+1}}.
\end{align*}
\begin{lemma}\label{lem:Diagonal}
  If $\theta \in (0,1)$, then
  \begin{align}\label{diag eqn}
   {\cD(N,\wh{f})} = {f(0)} +o(1)\quad\text{ and }\quad{\cD(N,|\wh{f}|)=O(1)} ,\quad N\to \infty 
   {\color{blue}.}
  \end{align}
\end{lemma}

\begin{proof}
    By a Riemann integral argument (see for example \cite[Theorem 3.2]{Apostol1976}) the sum
  \[
  \sum_{r \in \cR(k)} \frac{1}{r^{\Theta+1}}=\frac{\left(\alpha\theta k\right)^{-\Theta}-
  	\left(\alpha \theta k {(N+1)}^{\theta-1}\right)^{-\Theta}}{-\Theta}+O\left(\left(\alpha\theta k N^{\theta-1}\right)^{-\Theta-1}\right).
  \]
  
  Recall that $\abs{c_1}^2 = \Theta(\alpha\theta)^{\Theta}$. 
  Thus
  \begin{align*}
    {\cD(N,\wh{f})}
    = \frac{2}{N^2}\sum_{k \in [1,N^{1+\varepsilon})} 
    \wh{f}\Big(\frac{k}{N}\Big)
    \left( 
    {N} + O\left(\frac{N^{2-\theta}}{k}\right)\right)
    &= \frac{2}{N} \sum_{k \in [1,N^{1+\varepsilon})} \wh{f}\Big(\frac{k}{N}\Big) + O(N^{-\theta+\varepsilon}).
  \end{align*}
 Now \eqref{diag eqn} follows by the Poisson summation formula and the fact that $f$ is an even function.
 {The second statement of the lemma follows as above, only this time we employ the rapid decay of $\wh{f}$ to give an upper bound for $\cD(N,|\wh{f}|)$.}
\end{proof}

\subsection{Partial Summation}\label{ss:Partial Summation}

Lemma \ref{lem:E B} and Lemma \ref{lem:Diagonal} reduce the problem to estimating {$\cE^{(B)}(N,{g})-\cD(N,{g})$ for $g=\wh{f}$ and $|\wh{f}|$}. To estimate these errors requires a second application of the $B$-process, this time to the $k$-variable. In order to have adequate control on the derivative of $y\mapsto\gamma(\vect{r})y^\Theta$ in the next section, 
{we introduce a second 
decomposition}. In particular, let $U\in\mathbb{N}$ be such that $e^U\leq N^{1+\varepsilon}<e^{U+1}$. Then we may decompose the sum in $k$ into a sum over intervals $[e^{u},e^{u+1})$, the last one being $[e^U,N^{1+\epsilon})$. Let
\begin{align*}
  {\cE_{\vect{q},u}^{(B)}(N,g)}:= \frac{2}{N^2}\sum_{k \in [e^u,e^{u+1})} g\Big(\frac{k}{N}\Big) E_{q_1}^{(B)}(k) \overline{E_{q_2}^{(B)}(k)}.
\end{align*}
The next lemma reduces matters further to bounding the unweighted exponential sums 
\[
S(\gamma,\cI) :=\sum_{k\in \cI} e(\gamma k^\Theta)
\]
requiring the bound to be uniform in the size of $\gamma$, and the interval $\cI$. 
Therefore, we introduce
\[
\tilde{S}_{\Lambda_1,\Lambda_2}(u) := \sup_{\gamma \in [\Lambda_1, \Lambda_2]}  \,
\sup_{\cI \subseteq [e^u,e^{u+1})} \vert S(\gamma,\cI) \vert.
\]
 With this maximal operator at hand, we have the following

\begin{lemma}\label{lem:partial summation app}
  Let $q_1\leq q_2$ { and $g=\wh{f}$ or $|\wh{f}|$}. Then
  \begin{align}\label{part sum new 2}
      {\cE_{\vect{q},u}^{(B)}(N,g)}
      =
      1_{\lbrace q_1=q_2\rbrace}{D_{u,q_1}(N,g)} +
      O\left(N^{-2+\varepsilon} 
  \sum_{j \in \cJ_{u,q_1}}
  e^j  \sqrt{N_{q_2}^{\theta} N_{q_1}^{2-\theta}}
  \tilde{S}_{\Lambda_1(j),\Lambda_2(u)}(u)\right),
  \end{align}
  {where $1_{\lbrace P\rbrace}$ is $1$ if the property $P$ is satisfied and $0$ otherwise,}
  \[{D_{u,q}(N,g)}:= 
  \frac{2\abs{c_1}^2}{N^2} \sum_{k \in [e^{u},e^{u+1})} 
  g\Big(\frac{k}{N}\Big) \sum_{r \in \cR_q(k)} \frac{k^{\Theta}}{r^{\Theta+1}},
  \]
  $\Lambda_1(j):= C_1 e^{j-u\Theta}N_{q_1}$, $\Lambda_2(u) := C_2e^{u(1-\Theta)}N_{q_2}^\theta$ and $\cJ_{u,q_1} := [0,C_3 + u - (1-\theta) \Gamma \log q_1)$ for some constants $C_1,C_2,C_3>0$.
\end{lemma}

\begin{proof}

  For brevity, in this proof, let
  $$
  \gamma(\vect{r}) := \beta (r_2^{1-\Theta } -r_1^{1-\Theta}).
  $$  
  {We also consider only the case $g=\wh{f}$ since the second case follows by repeating the same arguments.}
  With this notation we have
  \begin{align*}
    {\cE_{\vect{q},u}^{(B)}(N,\wh{f}) }= 
    \frac{2\abs{c_1}^2}{N^2}\sum_{k\in [e^u,e^{u+1})} 
    \wh{f}\Big(\frac{k}{N}\Big) \sum_{\substack{r_i\in \cR_{q_i}(k)\\ i=1,2}} 
    \frac{k^{\Theta}}{(r_1r_2)^{\frac{\Theta+1}{2}}} e(-\gamma(\vect{r}) k^\Theta).
  \end{align*}
  Thus, the $r_i$ which appear in the overall sum all fall within the ranges
  \begin{align}
  \cR_{q_i,u}:=
  \bigg(
  \frac{\alpha\theta e^{u}}
       {N_{q_{i}+1}^{1-\theta}},
  \frac{\alpha\theta e^{u+1}}{N_{q_{i}}^{1-\theta}}
  \bigg)
  \qquad (i=1,2).
  \end{align}
  Now we interchange the $r$ and $k$ summations. For each choice of $r_1$ and $r_2$, we have that 
  \begin{align*}
    k \in \cK_{\vect{q}}(\vect{r}) := \frac{1}{\alpha \theta}
    \left[\max( 
r_{1}N_{q_{1}}^{1-\theta} , 
r_{2}N_{q_{2}}^{1-\theta},\alpha \theta e^u ),
\min (
r_{1}N_{q_{1}+1}^{1-\theta}, 
r_{2}N_{q_{2}+1}^{1-\theta}, \alpha \theta e^{u+1})\right).
  \end{align*}
 {Note that} this interval may sometimes be empty. With that,
  \begin{align*}
   {\cE_{\vect{q},u}^{(B)}(N,\wh{f}) } = \frac{2\abs{c_1}^2}{N^2}\sum_{\substack{r_i\in\cR_{q_i,u}\\i=1,2}}(r_1r_2)^{-\frac{\Theta+1}{2}}  \sum_{k \in \cK_{\vect{q}}(\vect{r})}\wh{f}\Big(\frac{k}{N}\Big)k^{\Theta} e(-\gamma(\vect{r}) k^\Theta).
  \end{align*}  
  Next, we remove the weights via partial summation, Lemma \ref{lem:partial summation}. 
  Let $w_k = \wh{f}(\frac{k}{N}) k^{\Theta}$, we first show 
\begin{equation}
  \left|w_k-w_{k+1}\right|\ll \frac{N^{\varepsilon} e^{\Theta u}}{k},\label{eq: difference of weights smaller}
\end{equation}
for any $k \in \mathcal{K}_{\vect{q}}(\vect{r})$ with the implied constant being absolute. 
The mean value theorem implies
\begin{align*}
\left|w_k-w_{k+1}\right| \leq
\left|\wh{f}\Big(\frac{k}{N}\Big)\right|
\left|k^{\Theta}-(k+1)^{\Theta}
\right|+\left|\wh{f}\Big(\frac{k}{N}\Big)-\wh{f}\left(\frac{k+1}{N}\right)
\right|(k+1)^{\Theta}
\ll \frac{|w_k|}{k} + \frac{e^{\Theta u}}{N}.
\end{align*}
Using that $k\ll N^{1+\varepsilon}$, yields \eqref{eq: difference of weights smaller}.

Note that $ K_{\mathbf{q}}(\mathbf{r}) \subset [e^u, e^{u+1})$. Thus Lemma \ref{lem:partial summation}
is applicable and, in combination with 
\eqref{eq: difference of weights smaller}, yields
\begin{align*}
  \sum_{k\in\mathcal{K}_{\mathbf{q}}(\mathbf{r})}
  w_k e(-\gamma (\vect{r})k^\Theta)
  \ll 
  N^{\varepsilon}e^{\Theta u}
  \max_{K \in \cK_{\vect{q}}(\vect{r})}\left|
  \sum_{k \in \cK_{\vect{q}}(\vect{r}) \cap [1,K]}
e(\gamma (\vect{r})k^\Theta)\right|.
\end{align*}
Note that each $r_i\in \cR_{q_i,u}$ satisfies 
$r_i \asymp e^uN_{q_i}^{-(1-\theta)}$ for $i=1,2$.

To reduce matters to exponential sums requires control of the difference 
$\gamma(\vect{r})$, thus let $\cT(j)$ denote the set of 
pairs $(r_1,r_2)\in \mathcal{R}_{q_{1},u} \times \mathcal{R}_{q_{2},u}$
satisfying $e^j\leq\vert  r_2 - r_1\vert <e^{j+1}$. 
Recall $q_1\le q_2$. From 
now on we assume $ r_2 < r_1$, since
the case $r_2 > r_1$ can be done in exactly 
the same way and therefore 
shall not be discussed in detail.
{Observe also that $q_1<q_2$ and $r_1=r_2$ implies that $\cK_{\vect{q}}(\vect{r})=\emptyset$, i.e. this case gives zero contribution.}
Assume now $q_1 < q_2$. Then
\begin{align*}
\gamma(\mathbf{r}) = \beta(\Theta-1)\int_{r_2}^{r_1} \tau^{-\Theta} \mathrm{d}\tau 
\gg (r_1 -r_2) r_1^{-\Theta}
\gg (r_1 -r_2) \left(\frac{N_{q_1}^{1-\theta}}{ \alpha\theta e^{u+1}}\right)^{\Theta}
\gg(r_1 -r_2) e^{-u \Theta} N_{q_1}.
\end{align*}
On the other hand,
\begin{align*}
\gamma(\mathbf{r})=\beta (r_2^{1-\Theta } -r_1^{1-\Theta})
&\ll\max\left(r_1^{1-\Theta},r_2^{1-\Theta}\right)=\min(r_1,r_2)^{1-\Theta}\\
&\ll\min\left(e^uN_{q_1}^{-(1-\theta)},e^uN_{q_2}^{-(1-\theta)}\right)^{1-\Theta}=e^{u(1-\Theta)}N_{q_2}^\theta.
\end{align*}
Thus, provided $\vect{r} \in \cT(j)$ we deduce that $\gamma (\vect{r}) \in [ \Lambda_1(j),\Lambda_2(u)]$.

Moreover, the range of $j$ can be constrained by
\begin{align*}
    e^j \asymp r_1 - r_2 \ll r_1 \ll e^u N_{q_1}^{-(1-\theta)}
\end{align*}
which implies that $j \in \cJ_{u,q_1}$. Thus, we have the following bound
\begin{align*}
   {\cE_{\vect{q},u}^{(B)}(N,\wh{f}) }
  &\ll 
  N^{-2+\varepsilon}\sum_{j \in \cJ_{u,q_1}}\Big(\frac{e^{u}}{ N_{q_1}^{1-\theta}} 
  \frac{e^{u}}{ N_{q_2}^{1-\theta}}\Big)^{-\frac{\Theta+1}{2}} e^{\Theta u}
  \#  \cT(j) \tilde{S}_{\Lambda_1(j),\Lambda_2(u)}(u)\\
  &\ll
  N^{-2+\varepsilon} 
  \sum_{j \in \cJ_{u,q_1}}
  \Big(\frac{e^{u}}{ N_{q_1}^{1-\theta}} 
  \frac{e^{u}}{ N_{q_2}^{1-\theta}}\Big)^{-\frac{\Theta+1}{2}}
  e^{\Theta u}
  \frac{e^{u}}{ N_{q_2}^{1-\theta}} e^{j}
  \tilde{S}_{\Lambda_1(j),\Lambda_2(u)}(u).
\end{align*}
We observe that
\begin{align*}
  \Big(\frac{e^{u}}{ N_{q_1}^{1-\theta}} 
  \frac{e^{u}}{ N_{q_2}^{1-\theta}}\Big)^{-\frac{\Theta+1}{2}}
  e^{\Theta u}
  \frac{e^{u}}{ N_{q_2}^{1-\theta}} e^{j}
  =
  e^j \sqrt{N_{q_2}^{\theta} N_{q_1}^{2-\theta}}.
\end{align*}
Overall we infer, for $q_1<q_2$, that
\begin{align*}
    {\cE_{\vect{q},u}^{(B)}(N,\wh{f}) }
    &\ll  
    N^{-2+\varepsilon} 
  \sum_{j \in \cJ_{u,q_1}}
  e^j \sqrt{N_{q_2}^{\theta} N_{q_1}^{2-\theta}}
 \tilde{S}_{\Lambda_1(j),\Lambda_2(u)}(u).
\end{align*} 
If $q_1 = q_2$, we remove the diagonal $r_1 = r_2$, which corresponds to the term ${D_{u,q_1}(N,\wh{f})}$,
and apply exactly the same bound to the off-diagonal. 
\qedhere

\end{proof}

\section{Proof of the Main Theorem}
\label{s:Proof of the Main Theorem}

\subsection{Exponential Sum Bounds}
\label{ss:Exponential Sum Bounds}
Thanks to Lemmas \ref{lem:E B}, \ref{lem:Diagonal}, and \ref{lem:partial summation app}, 
we will show that Theorem \ref{thm: main theorem} can be deduced from the next lemma.

\begin{lemma}\label{lem: maximal operator estimates} 

Assume $q_1 \leq q_2$. 
Then we have the following bound
  \begin{equation}\label{eq: max op bound 2}
    \tilde{S}_{\Lambda_1(j),\Lambda_2(u)}(u)\ll   e^{8u/15} N_{q_2}^{11\theta/30} + e^{u-j/2}N_{q_1}^{-1/2}.
  \end{equation}
\end{lemma}

\begin{proof}
 First we apply the B-process (Theorem \ref{thm:B-Process}) in the $k$ variable. $k$ is in an interval $[A,B]$ of size $e^u$, and the phase function 
 $\phi(k) = \gamma k^{\Theta}$ 
 satisfies 
 $\phi(k) \asymp \gamma e^{u\Theta}$,
 as well as $\phi^\prime(k) \asymp \gamma
  e^{u(\Theta-1)}$ and $\phi^{\prime\prime}(k) 
  \asymp \gamma e^{u(\Theta-2)}$, 
  where $\gamma \in [ \Lambda_1(j),\Lambda_2(u)]$.
  Applying Theorem \ref{thm:B-Process} 
  and a trivial estimate gives one bound 
  (which will suffice for $\theta < 1/3$).

    In fact, we require slightly more than the B-process to move past $\theta = 1/3$. Thus we will apply the B-process precisely, and then use partial summation and the A-process to bound the resulting sum. First, set $\vartheta = 1/(1-\Theta)$ and apply Theorem \ref{thm:B-Process}, to conclude 
    \begin{align}\label{B process app S}
    \sum_{k\in[e^u,e^{u+1})}e(\gamma k^{\Theta})=c_3\sum_{h\in[a,b)}\sqrt{\frac{\gamma^{\vartheta}}{h^{\vartheta+1}}}\,e(c_4\gamma^{\vartheta} h^{1-\vartheta})+O(e^{u-j/2}N_{q_1}^{-1/2}+\log(N_{q_2}))
    \end{align}
    where $a<b$ are positive integers of size $\gamma e^{u(\Theta -1)}$ and $c_3,c_4$ are (complex, respectively real) nonzero constants.
     A trivial estimate implies
    that we can assume $b-a \ge 10$.

    By exploiting partial summation, we may apply Lemma \ref{lem:partial summation} to the main term in \eqref{B process app S}. Thus, to prove \eqref{eq: max op bound 2}, it suffices to bound:
    \begin{equation}\label{eq: k summation unweighted intermediate claim}
        \frac{e^{u}}{\gamma^{1/2}
         e^{\Theta u/2 }}\sum_{h\in[a,b)}
         e(c_4\gamma^{\vartheta} h^{1-\vartheta}) 
    \end{equation}
    for $a<b$ being of size 
    $\gamma e^{u(\Theta -1)}$ 
    and such that $b-a\geq 10$. 
    To that end, we use 
    Theorem \ref{thm:A-Process} 
    for an arbitrary integer $l$. 
    In that notation, let 
    $F \asymp \gamma e^{u\Theta}$, 
    and $M \asymp \gamma e^{u(\Theta -1)}$.
    Thus, we conclude that (recall $L = 2^l$)
    \begin{align}
    \begin{aligned}
        \sum_{h\in[a,b)}
        e(c_4 \gamma^{\vartheta} h^{1-\vartheta}) 
        &\ll 
        (\gamma e^{u\Theta})^{\frac{1}{4L-2}}
        (\gamma 
        e^{u(\Theta -1)})^{1-\frac{l+2}{4L-2}} 
        +  e^{-u}.
    \end{aligned}
    \end{align}
    Inserting this into \eqref{eq: k summation unweighted intermediate claim} and using that $e^{j-u\Theta}N_{q_1}\ll\gamma \ll e^{u(1-\Theta)}N_{q_2}^\theta$
    
    \begin{align*}
        \frac{e^{u}}{\gamma^{1/2} 
        e^{\Theta u/2 }}\sum_{h\in[a,b)}
        e(c_4\gamma^{\vartheta} h^{1-\vartheta})
        &\ll
        e^u (e^uN_{q_2}^\theta)^{\frac{1}
        {4L-2}-1/2}(N_{q_2}^\theta)^{1-
        \frac{l+2}{4L-2}}
        +e^{-j/2}N_{q_1}^{-1/2},
    \end{align*}
    which, on rearranging and choosing $l=3$, gives \eqref{eq: max op bound 2}.
    
\end{proof}

\subsection{Proof of Theorem \ref{thm: main theorem}}
\label{ss:Proof of Main}

Recall that, after applying Lemma \ref{lem:E B}, our goal is to show
$$
{\cE^{(B)}(N,\wh{f}) }=\sum_{\vect{q}\in[1,Q]^2}\sum_{u\leq U}{\cE_{\vect{q},u}^{(B)}(N,\wh{f}) }={f(0)}+o(1)\quad\text{ and }\quad{\cE^{(B)}(N,|\wh{f}|) }=O(1),\quad N\to\infty.
$$



In view of \eqref{eq: max op bound 2} (a similar calculation shows that the second error term in \eqref{eq: max op bound 2} gives a negligible contribution) we have {for $q_1\leq q_2$} that
\begin{align*}
    N^{{\varepsilon-2}} \sqrt{N_{q_2}^{\theta} N_{q_1}^{2-\theta}} \sum_{j \in \cJ_{u,q_1}} e^{j} \tilde{S}_{\Lambda_1(j),\Lambda_2(u)}(u)
    &\ll
    N^{{\varepsilon-2}} \sqrt{N_{q_2}^{\theta} N_{q_1}^{2-\theta}} \sum_{j \in \cJ_{u,q_1}} e^{j} e^{8u/15} N_{q_2}^{11\theta/30} \\
    &\ll
    N^{{\varepsilon-2}} \sqrt{N_{q_2}^{\theta} N_{q_1}^{\theta}} e^{u} e^{8u/15} N_{q_2}^{11\theta/30}\\
    &\ll
    {N^{\frac{41\theta-14}{30}+\varepsilon}}.
\end{align*}
{The case $q_2<q_1$ can be treated in the same way by interchanging the roles of $q_1$ and $q_2$.}
Therefore, taking into account Lemma \ref{lem:Diagonal} and Lemma \ref{lem:partial summation app}, we obtain that

\begin{align*}
\sum_{\vect{q}\in[1,Q]^2}\sum_{u\leq U}{\cE_{\vect{q},u}^{(B)}(N,\wh{f}) }&={{\cD(N,\wh{f})}}+O\left(N^{\frac{{41\theta-14}}{30}+\epsilon}\sum_{\vect{q}\in[1,Q]^2}\sum_{u\leq U}1\right)\\
&={f(0)}+o(1)+O\left(N^{\frac{{41\theta-14}}{30}+\frac{2}{\Gamma}+\varepsilon}\right),\quad N\to\infty.
\end{align*}
In a similar fashion
\[
{\cE^{(B)}(N,|\wh{f}|)\ll1+N^{\frac{{41\theta-14}}{30}+\frac{2}{\Gamma}+\varepsilon}}.
\]
Thus, for any $\theta < 14/41$ the theorem follows by choosing $\Gamma$
large enough and $\varepsilon$ small enough.
\qed

   \section*{Acknowledgements}
   AS was supported by the Austrian Science Fund (FWF): project M 3246-N. NT was supported by a Schr\"{o}dinger Fellowship of the Austrian Science Fund (FWF): project J 4464-N. The authors would like to thank Jens Marklof and Zeev Rudnick for comments on a earlier draft and Christoph Aistleitner, Daniel El-Baz and Marc Munsch for discussions and comments. We are especially grateful to the anonymous {referees} for their thorough reading and many insightful comments. 
   
\section*{Funding}
This research was funded in whole, or in part, by the Austrian Science Fund (FWF) M 3246-N. For the purpose of open access, Athanasios Sourmelidis has applied a CC BY public copyright licence to any Author Accepted Manuscript version arising from this submission.

  \bibliographystyle{alpha}
  \bibliography{biblio}

    \hrulefill

    \vspace{4mm}
     \noindent Department of Mathematics, Rutgers University, Hill Center - Busch Campus, 110 Frelinghuysen Road, Piscataway, NJ 08854-8019, USA. \emph{E-mail: \textbf{chris.lutsko@rutgers.edu}}
     
         \vspace{4mm}
     \noindent Institute of Analysis and Number Theory, Graz University of Technology, Steyrergasse 30, 8010 Graz, Austria. \emph{E-mail: \textbf{sourmelidis@math.tugraz.at}}

\vspace{4mm}
     \noindent
Department of Mathematics, 
University of Wisconsin, 480 Lincoln Drive, Madison, WI, 53706, USA
\emph{E-mail: \textbf{ technau@wisc.edu}}

\end{document}